\documentclass[12 pt]{article}%
\usepackage{amsmath, amsfonts, amsthm, color,latexsym}
\usepackage{amsmath}
\usepackage{amsfonts}
\usepackage{amssymb}
\usepackage{color, soul}
\usepackage[all]{xy}
\usepackage{graphicx}%
\providecommand{\U}[1]{\protect\rule{.1in}{.1in}}
\allowdisplaybreaks[4]
\newtheorem{theorem}{Theorem}[section]
\newtheorem{proposition}[theorem]{Proposition}
\newtheorem{corollary}[theorem]{Corollary}
\newtheorem{example}[theorem]{Example}

\newtheorem{remark}[theorem]{Remark}

\newtheorem{lemma}[theorem]{Lemma}
\newtheorem{final remark}[theorem]{Final Remark}
\newtheorem{definition}[theorem]{Definition}
\allowdisplaybreaks[4]

\begin{document}

\title{On the Schur, positive Schur and weak Dunford-Pettis properties in Fr\'echet lattices}
\author{Geraldo Botelho\thanks{Supported by CNPq Grant
304262/2018-8 and Fapemig Grant PPM-00450-17.}\,\, and  Jos\'e Lucas P. Luiz\thanks{Supported by a CNPq scholarship\newline 2010 Mathematics Subject Classification: 46A40, 46B42, 46A04.\newline Keywords: Banach lattices, Fr\'echet lattices, Schur and positive Schur properties, dual positive Schur property, weak Dunford-Pettis property.
}}
\date{}
\maketitle

\begin{abstract} We prove some general results on sequential convergence in Fr\'echet lattices that yield, as particular instances, the following results regarding a closed ideal $I$ of a Banach lattice $E$: (i) If two of the lattices $E$, $I$ and $E/I$ have the positive Schur property (the Schur property, respectively) then the third lattice has the positive Schur property (the Schur property, respectively) as well; (ii) If $I$ and $E/I$ have the dual positive Schur property, then $E$ also has this property; (iii) If $I$ has the weak Dunford-Pettis property and $E/I$ has the positive Schur property, then $E$ has the weak Dunford-Pettis property. Examples and applications are provided.
\end{abstract}

\section{Introduction}

In the realm of Banach spaces, the Schur property (weakly null sequences are norm null) is a 3-space property in the weak sense that a Banach space $E$ has the Schur property whenever a closed subspace $F$ of $E$ and the quotient space $E/F$ have the Schur property (see, e.g., \cite{castillo}). But it is not a 3-space property in the strong sense that, given a closed subspace $F$ of the Banach space $E$, if two of the spaces $E$, $F$ and $E/F$ have the Schur property, then the third one also has this property. To see that, just remember that $c_0$ is a quotient of $\ell_1$.

In the setting of Banach lattices, for the quotient $E/F$ of a Banach lattice $E$ over a closed subspace $F$ to be a Banach lattice, $F$ should be an ideal of $E$ (see, e.g., \cite{aliprantis2}). So the natural transposition of the concept of 3-space property (in the strong sense) to the realm of Banach lattices reads as follows.

\begin{definition}\rm A property $\cal P$ of Banach lattices is a {\it 3-lattice property} if the following holds: given a closed ideal $I$ of the Banach lattice $E$, if two of the lattices $E$, $I$ and $E/I$ have $\cal P$, the the third one has $\cal P$ as well.
\end{definition}

A Banach lattice has the positive Schur property if weakly null sequences formed by positive vectors are norm null. A lot of research has been done on this property, for some recent contributions see, e.g., \cite{ardakani, baklouti, kaminska, moussa, tradacete, wnuk, wnuk2009, zeekoei}. Among other results, in this paper we prove that, contrary to the case of the Schur property for Banach spaces, the Schur and the positive Schur properties are 3-lattice properties. These results will appear as applications of general results on the Schur and the positive Schur properties in Fr\'echet lattices.

\begin{definition}\rm Given linear topologies $\tau_1$ and $\tau_2$ in a Riesz space $E$, we say that:\\
$\bullet$ $E$ has the {\it $(\tau_1, \tau_2)$-Schur property} (in short, $(\tau_1, \tau_2)$-SP) if $\tau_1$-null sequences in $E$ are $\tau_2$-null.\\
$\bullet$ $E$ has the {\it $(\tau_1, \tau_2)$-positive Schur property} (in short, $(\tau_1, \tau_2)$-PSP) if $\tau_1$-null sequences in $E$ formed by positive elements are $\tau_2$-null.
\end{definition}

Of course, the Schur property is the (weak topology, norm topology)-SP and the positive Schur property is the (weak topology, norm topology)-PSP in a Banach lattice.  Part (a) of the definition above is similar to the approach of Castillo and Sim\~oes to the Schur property (see \cite[Definition 4, Definition 5, Remark 5]{castillo1}).

After proving some general results on the $(\tau_1, \tau_2)$-SP and the $(\tau_1, \tau_2)$-PSP (see Theorems \ref{implret} and \ref{inequot}), we conclude in Corollary \ref{corcor} that the Schur and the positive Schur properties are 3-lattices properties in the context of $\sigma$-Dedekind complete Fr\'echet lattices. Applying these results for Banach lattices we get, in Theorem \ref{3latpro}, that the Schur and the positive Schur properties are 3-lattice properties, and in Proposition \ref{propdpsp} that a Banach lattice $E$ has the dual positive Schur property whenever a closed ideal $I$ and the quotient lattice $E/I$ have this property. A few more applications to Banach lattices and an illustrative example are also provided. In a short final section we show that the methods we use for Fr\'echet lattices can also be helpful in the study of the weak Dunford-Pettis property in Banach lattices (cf. Theorem 3.2 and Corollary 3.3).  

For the basic theory of Banach lattices, Riesz spaces and linear topologies on Riesz spaces we refer \cite{aliprantis, aliprantis1, aliprantis2, meyer}

\section{The Schur and the positive Schur properties}


According to \cite{aliprantis1}, a {\it Fr\'echet lattice} is a Riesz space endowed with a locally convex-solid (hence linear), metrizable and complete topology. We just recall that a locally convex topology on a Riesz space $E$ is locally convex-solid if it is generated by a family of Riesz seminorms, or, equivalently, if the origin has a basis of neighborhoods formed by convex and solid sets (see \cite[Section 3.3]{aliprantis2}).

  If $F$ is a Riesz subspace of the Risez space $E$ endowed with the linear topologies $\tau_1$ and $\tau_2$, we consider in $F$ the corresponding relative topologies, still denoted by $\tau_1$ and $\tau_2$. So, if $E$ has the $(\tau_1, \tau_2)$-SP ($(\tau_1, \tau_2)$-PSP, respectively), then $F$ has the $(\tau_1, \tau_2)$-SP ($(\tau_1, \tau_2)$-PSP, respectively)  as well.

Given a subspace $M$ of a topological vector space $(X, \tau)$, by
$\dot \tau$ we denote  the quotient topology on $X/M$, which is the finest topology on $X/M$ making the quotient operator $\pi \colon X \longrightarrow X/M$ continuous. It is well known that $\dot \tau$ is a linear topology on $X/M$ (see, e.g., \cite[I.2.2]{schaefer}).

Given an ideal $I$ of the Riesz space $E$, the order in $E/I$ is given by: $\dot x\leq \dot y$ if there are $x_1\in \dot x$ and $y_1\in\dot y$ such that $x_1\leq y_1$ (see \cite[pp. 99, 100]{aliprantis2} and \cite[p. 16]{aliprantis1}).

The symbol $x_n \stackrel{\tau}{\longrightarrow}x$ means that the sequence $(x_n)_n$ converges to $x$ with respect to the topology $\tau$.

\begin{theorem}\label{implret}
Let $\tau_1$ and $\tau_2$ be linear topologies on the Riesz space $E$ such that  $(E,\tau_2)$ is a Fr\'echet lattice and $E$ has the $(\tau_2,\tau_1)$-SP. Let $I$ be a 
$\tau_2$-closed ideal of $E$ such that $I$ has the $(\tau_1,\tau_2)$-PSP ($(\tau_1,\tau_2)$-SP, respectively) and $E/I$ has the $(\dot\tau_1,\dot\tau_2)$-PSP ($(\dot \tau_1,\dot \tau_2)$-SP, respectively). Then $E$ has the $(\tau_1,\tau_2)$-PSP ($(\tau_1,\tau_2)$-SP, respectively).
\end{theorem}

\begin{proof}
Let us prove first the case of the $(\tau_1,\tau_2)$-PSP.
Let $(x_n)_n\subseteq E$ be a sequence such that $x_n \geq 0$ for every $n$ and $x_n \stackrel{\tau_1}{\longrightarrow}0$, and let $(x_{n_j})_j$ be an arbitrary subsequence of $(x_n)_n$. The continuity of the quotient operator $\pi\colon (E, \tau_1)\longrightarrow (E/I, \dot \tau_1)$ gives
$$\dot x_{n_j}:=\pi(x_{n_j})\stackrel{\dot\tau_1}{\longrightarrow}0 {\rm ~in~} E/I,$$
and the fact that $\pi$ is a Riesz homomorphism  \cite[Theorem 2.22]{aliprantis2} gives $\dot x_{n_j} \geq 0$ for every $n$. Calling on the $(\dot \tau_1, \dot \tau_2)$-PSP of  $E/I$, we get ${\dot{x}_{n_j}}\stackrel{\dot \tau_2}{\longrightarrow}0 $ in $E/I$. Since $(E,\tau_2)$ is a Fr\'echet lattice, in particular $\tau_2$ is a metrizable topology, hence it is generated by an invariant metric $d\colon E\times E\longrightarrow [0,\infty)$ (see \cite[Theorem 5.10]{aliprantis}). Therefore, the quotient topology  $\dot \tau_2$ in $E/I$ is metrizable as well, and, in particular, it is generated by the invariant metric
\begin{align*}
\dot{d}\colon E/I\times E/I &\longrightarrow [0,\infty)~,~\dot{d}(\dot{x},\dot{z}):=\inf\{d(x-z,y): y\in I\},
\end{align*}
(see \cite[Theorem 1.41]{rudin}).
Since ${\dot{x}_{n_j}}\stackrel{\dot \tau_2}{\longrightarrow}0$ if and only if  
$\inf\{d(x_{n_j},y): y\in I\}\longrightarrow 0$, for every  $m\in\mathbb{N}$ there exists $ N_m\in\mathbb{N}$ such that
 $$\inf\{d(x_{n_j},y): y\in I\}<\frac{1}{m} \textrm{~for every~} j\geq N_m.
$$
  So, for $m=1$ there exists  $N_1\in\mathbb{N}$ and $y_{N_1}\in I$ such that $d(x_{n_{N_1}}, y_{N_1})<1$, for  $m=2$ there exists $N_2>N_1$ and $y_{N_2}\in I$ such that $d(x_{n_{N_2}}, y_{N_2})<\frac{1}{2}$, and so on.
%
In this fashion we construct a sequence $(y_{N_m})_m$ in $I$ and a subsequence $(x_{n_{N_m}})_m$ of $(x_{n_j})_j$ such that
$$d(x_{n_{N_m}}- y_{N_m},0)=d(x_{n_{N_m}}, y_{N_m})<\frac{1}{m} {\rm ~for~every~} m\in\mathbb{N}.$$ It follows that $x_{n_{N_m}}- y_{N_m}\stackrel{\tau_2}{\longrightarrow} 0$ em $E$.
As $\tau_2$ is a metrizable locally convex-solid topology on $E$, 
combining \cite[Lemma 5.75]{aliprantis} and \cite[Theorem 2.25]{aliprantis1} we conclude that $\tau_2$ is generated by a sequence $(\rho_k)_k$ of Riesz seminorms. This gives that
$\rho_k(x_{n_{N_m}}- y_{N_m})\longrightarrow 0$ for every $k\in\mathbb{N}$, and applying the inequality $|x_{n_{N_m}}- y_{N_m}^+|\leq |x_{n_{N_m}}- y_{N_m}|$ (see \cite[p. 8]{aliprantis2}), where $ y_{N_m}^+$ is the positive part of  $y_{N_m}$, we get
$$\rho_k(x_{n_{N_m}}- y_{N_m}^+)\longrightarrow 0 {\rm ~for~ every~} k\in\mathbb{N},$$
from which it follows that 
$$x_{n_{N_m}}-y_{N_m}^+\stackrel{\tau_2}{\longrightarrow} 0 {\rm ~in~} E.$$
The $(\tau_2, \tau_1)$-SP of $E$ gives  $x_{n_{N_m}}-y_{N_m}^+\stackrel{\tau_1}{\longrightarrow} 0$ in $E$, and from the linearity of $\tau_1$ we get $$y_{N_m}^+=x_{n_{N_m}}-(x_{n_{N_m}}-y_{N_m}^+)\stackrel{\tau_1}{\longrightarrow} 0 {\rm ~in~} E,$$ which gives $y_{N_m}^+\stackrel{\tau_1}{\longrightarrow}0$ em $I$.
 Since $I$ has the $(\tau_1, \tau_2)$-PSP and each $y_{N_m}^+ \in I$ is positive, we have $y_{N_m}^+\stackrel{\tau_2}{\longrightarrow} 0$ em $I$,
  hence $y_{N_m}^+\stackrel{\tau_2}{\longrightarrow} 0$ em $E$. Now the linearity of $\tau_2$ gives $$x_{n_{N_m}}=(x_{n_{N_m}}-y_{N_m}^+)-y_{N_m}^+\stackrel{\tau_2}{\longrightarrow}0 {\rm ~in~} E.$$
 We have proved that every subsequence of $(x_n)_n$ admits a further subsequence that is $\tau_2$-null. This is enough to conclude that $(x_n)_n$ is itself $\tau_2$-null, proving that $E$ has the $(\tau_1, \tau_2)$-PSP.

The proof of the case of the $(\tau_1, \tau_2)$-SP follows the same steps, actually it is a bit easier because there is no need to pass to the positive part of $y_{N_m}$.
\end{proof}

In the particular case for the Schur property where $\tau_1$ is the weak topology on $E$ with respect to the linear topology $(E,\tau_2)$, we have $\tau_1 \subseteq \tau_2$, hence  $E$ has the $(\tau_2,\tau_1)$-SP. This means that, for this particular choice of $\tau_1$, the part about the Schur property of the theorem above collapses to \cite[Proposition 6]{castillo1} (see also \cite[Theorem 6.1.a]{castillo}) due to J. Castillo and M.A. Sim\~oes, whose proof was our first inspiration.

To proceed we need the following lemma. Given a Riesz  space $E$ endowed with a linear topology $\tau$, by $w_{\tau}$ we denote the weak topology on $E$ with respect to the topology $\tau$. It is well known that $w_{\tau}$ is a linear topology on $E$ (see, \cite{narici, schaefer}).

\begin{lemma} \label{coplinf} If $(E,\tau)$ is a locally convex Hausdorff Riesz space with the   $(w_{\tau},\tau)$-PSP, then $(E,\tau)$ does not contain a lattice copy of  $\ell_\infty$, meaning that there is no Riesz homomorphism from $\ell_\infty$ to $E$ that is a homeomorphism onto its range.
\end{lemma}

\begin{proof} Suppose that there exists a Riesz homomorphism   $T\colon \ell_\infty\longrightarrow E$ that is a homeomorphism onto $T(\ell_\infty)$. As $\tau$ is locally convex, by \cite[Theorem 8.12.2]{narici} we know that the weak topology on $T(\ell_\infty)$ with respect to $\tau$ is the the weak topology $w_{\tau}$ on $E$ restricted to $T(\ell_\infty)$, which we still denote by $w_\tau$. Combining this with the $(w_{\tau},\tau)$-PSP of $E$, it follows that $T(\ell_\infty)$ has the $(w_{\tau},\tau)$-PSP as well. Since $T$ and $T^{-1} \colon T(\ell_\infty) \longrightarrow \ell_\infty$ are positive operators (obvious
) and weak-weak continuous \cite[Theorem 8.11.3(c)]{narici}, we have that $\ell_\infty$ has the PSP. This contradiction completes the proof (to see that $\ell_\infty$ fails the PSP, note, e.g., that $c_0$ is a Banach sublattice of $\ell_\infty$ \cite[p. 12, Example iii)]{meyer} and the canonical unit vectons in $c_0$ form a non-norm null weakly null sequence of positive vectors in $c_0$).
\end{proof}

Remember that if $(E,\tau)$ is a Riesz space with a linear topology, then the weak topology $w_{\dot \tau}$ on $E/I$ with respect to the space  $(E/I, \dot \tau)$ coincides with the quotient topology $\dot {w_{\tau}}$ with respect to the space $(E, w_{\tau})$ (see \cite[Theorem 8.12.3(a)]{narici}).

Remember also that a Riesz space $E$ has the {\it projection property} if every band in $E$ is a projection band; and a locally solid Riesz space $(E, \tau)$ has the {\it Lebesgue property} if $x_\alpha\stackrel{\tau}{\longrightarrow}0$ whenever $x_\alpha\downarrow 0$ in $E$ (here $(x_\alpha)_\alpha$ is a net in $E$). More details can be found in \cite{aliprantis1}.

\begin{theorem}\label{inequot}
 Let $\tau$ be a locally convex-solid Hausdorff complete topology on the Dedekind $\sigma$-complete Riesz space $E$. If $E$ has the $(w_\tau,\tau)$-PSP ($(w_\tau,\tau)$-SP, respectively), then, regardless of the $\tau$-closed ideal $I$ of $E$, the quotient space $E/I$ has the $(w_{\dot \tau}, \dot \tau)$-PSP ($(w_{\dot \tau}, \dot \tau)$-SP, respectively).
\end{theorem}

\begin{proof} Assume that $E$ has the $(w_\tau,\tau)$-PSP. 
As $\tau$ is locally convex, by Lemma \ref{coplinf} we know that $(E,\tau)$ does not contain a lattice copy of $\ell_\infty$. Theorem 3.29 in \cite{aliprantis1} tells us that $(E,\tau)$ has the Lebesgue property, thus $E$ is Dedekind complete by  \cite[Theorem 3.24]{aliprantis1}. We conclude that $E$ has the projection property by \cite[Theorem 1.42]{aliprantis2}.

Let $I$ be a $\tau$-closed ideal of $E$. Since $(E,\tau)$ is locally solid with the Lebesgue property, calling \cite[Theorem 3.7]{aliprantis1} we have that $I$ is a band, hence a projection band, which means that $E=I\oplus I^d$, where $
I^d=\{x\in E: ~x\perp y ~\textrm{for every}~y\in I\}$. Using again that the weak topology on $I^d$ with respect to the $\tau$ is the restriction of the weak topology $w_\tau$ on $E$, the $(w_\tau,\tau)$-PSP of $E$ implies that $I^d$ has the $(w_\tau,\tau)$-PSP too. 
It follows from \cite[p. 115]{aliprantis1} that $(E/I,\dot \tau)$ is lattice isomorphic to $(I^d,\tau)$, so 
$E/I$ has the $(w_{\dot \tau}, \dot \tau)$-PSP.

The case of the $(w_\tau,\tau)$-SP is similar. First observe that the $(w_\tau,\tau)$-SP is inherited by Riesz subspaces and is preserved by isomorphisms between locally convex spaces. Next follow the steps above and, as $(E/I,\dot \tau)$ is lattice isomorphic to $(I^d,\tau)$ and this latter space has the $(w_\tau,\tau)$-SP, then $E/I$ has the  $(w_{\dot \tau}, \dot \tau)$-SP.
\end{proof}

As to the assumptions of the theorem above, recall that, as in the case of Banach spaces, a linear subspace of a locally convex space is closed if and only if it is weakly closed \cite[Theorem 8.8.1]{narici}.

Combining Theorems \ref{implret} and \ref{inequot} we get the following:

\begin{corollary}\label{corcor} The (weak, strong)-positive Schur property and the (weak, strong)-Schur property are  3-lattice properties in the context of $\sigma$-Dedekind complete Fr\'echet lattices in the following sense: if $I$ is a closed ideal of a $\sigma$-Dedekind complete Fr\'echet lattice $(E, \tau)$ such that two out of the three following conditions hold: \\
\indent{\rm (i)} $E$ has the $(w_\tau, \tau)$-PSP ($(w_\tau, \tau)$-SP, respectively),\\
\indent{\rm (ii)} $I$ has the $(w_\tau, \tau)$-PSP  ($(w_\tau, \tau)$-SP, respectively),\\
\indent{\rm (iii)} $E/I$ has the $(\dot w_\tau, \dot \tau)$-PSP ($(\dot w_\tau, \dot \tau)$-SP, respectively);\\
then the third condition holds too.
\end{corollary}

The notion of 3-lattice properties in Banach lattices was defined in the Introduction.

\begin{theorem}\label{3latpro} The Schur and the positive Schur properties are 3-lattice properties.
\end{theorem}

\begin{proof} Since $c_0$ fails the positive Schur property, Banach lattices with the Schur/positive Schur property do not contain a copy of $c_0$, hence they are KB-spaces \cite[Theorem 4.60]{aliprantis2}. But KB-spaces have order continuous norms \cite[p. 232]{aliprantis2} and Banach lattices with order continuous norms are Dedekind complete \cite[Corollary 4.10]{aliprantis2}. So, Banach lattices with the Schur/positive Schur property are Dedekind complete. Of course, in Banach lattices the Schur property is the (weak, norm)-SP and the positive Schur property is the (weak,norm)-PSP, so the result follows from Corollary \ref{corcor}.
\end{proof}

Next we give applications and examples regarding the results obtained thus far.

\begin{proposition}
Let $E$ and $F$ be Banach lattices and $T\colon E\longrightarrow F$ be a Riesz homomorphism with closed range. If $\ker(T)$ and $F$ have the positive Schur property (Schur property, respectively), then $E$ has the positive Schur property (Schur property, respectively).
\end{proposition}

\begin{proof} Let us prove the case of the positive Schur property. As a Banach sublattice of $F$, $T(E)$ has the positive Schur property. We believe it is well known that $E/\ker(T)$ is lattice isomorphic to $T(E)$, but since we have found no reference to quote, we give a short reasoning. Since $T$ is a Riesz homomorphism,   $\ker(T)$ is an ideal de $E$ \cite[p. 94]{aliprantis2}), and the continuity of $T$ (Riesz homomorphisms are positive hence continuous) guarantees that $\ker(T)$ is closed. So $E/\ker(T)$ is a Banach lattice. The operator
$$ S\colon E/\ker(T)\longrightarrow T(E)~,~S(\dot x) =T(x),$$
is an isomorphism between Banach spaces such that $S\circ \pi=T$, where $\pi\colon E\longrightarrow E/\ker(T)$ is the quotient operator \cite[Theorem 1.7.14]{megginson}. From the fact that $T$ and $\pi$ are Riesz homomorphisms it follows easily that $S$ is a Riesz homomorphism as well, so $E/\ker(T)$ is lattice isomorphic to $T(E)$. Therefore $E/\ker(T)$ has the positive Schur property and Theorem \ref{3latpro} gives that $E$ has the positive Schur property. The case of the Schur property is identical.
\end{proof}

On the one hand, higher order duals $E^{**}, E^{***}, \ldots$ of infinite dimensional Banach spaces $E$ never have the Schur property (this follows from \cite[Corollary 11]{mujica}). On the other hand, the bidual of an infinite dimensional Banach lattice with the positive Schur property may have the positive Schur property. For example, every AL-space has the positive Schur property \cite[Examples 1.3]{rabiger} and the bidual of an AL-space is an AL-space as well \cite[Theorem 4.23]{aliprantis2}. We believe that, contrary to the case of the Schur property in Banach spaces, it is not easy to give concrete examples of Banach lattices $E$ with the positive Schur property such that $E^{**}$ fails this property. In this case, is $E$ a closed ideal of $E^{**}$? If yes, our results imply that $E^{**}/E$ fails the positive Schur property. Our next aim is to give such an example.

\begin{proposition}
\label{corquocbid} Let $E$ be a Banach lattice with the positive Schur property such that $E^*$ contains a lattice copy of $\ell_1$. Then $E$ is a closed ideal of $E^{**}$ and $E^{**}$  and $E^{**}/E$ fail the positive Schur property.
\end{proposition}
\begin{proof} We have already noticed that Banach latties with the positive Schur property do not contain a lattice copy of $c_0$, so $E$ does not contain a lattice copy of $c_0$, therefore $E$ is a KB-space (we have already used this fact). By  \cite[Theorem 4.60]{aliprantis2} we know that $E$ is a band, in particular a closed ideal, of $E^{**}$. So we can consider the quotient lattice $E^{**} /E$. Since $E^*$ contains a lattice copy of $\ell_1$, from a result due to Wnuk \cite[p. 22]{wnuk} it follows that $E^{**}$ fails the positive Schur property. Corollary \ref{3latpro} gives that $E^{**}/E$ lacks the positive Schur property.
\end{proof}

\begin{example}\rm
  For each $n\in\mathbb{N}$ let $\ell_n^\infty$ denote the Banach lattice $\mathbb{R}^n$ endowed with the maximum norm and coordinatewise order. Now let $E$ denote the $\ell_1$-sum of the sequence $(\ell_n^\infty)_n$, that is
$$
E:=\left( \bigoplus\limits_{n\in\mathbb{N}}\ell_n^\infty\right)_1=\left\{x=(x_n)_n : x_n\in \ell_n^\infty \textrm{~for every~}n\in\mathbb{N} \textrm{~and~} \|x\|:=\displaystyle\sum_{n=1}^\infty\|x_n\|<\infty\right\},
$$
which is a Banach lattice endowed with the coordinatewise order. As a finite dimensional Banach lattice, each $\ell_n^\infty$ has the positive Schur property, so $E$ has the positive Schur property by another result due to Wnuk \cite[p. 17]{wnuk}.

Let us see now that $E^*$ contains a lattice copy of $\ell_1$. To do so, first remember that $(\ell_n^\infty)^*$ is canonically lattice isometric to $\ell_n^1 = (\mathbb{R}^n, \|\cdot\|_1)$. So, by \cite[Theorem 4.6]{aliprantis2} we know that $E^*$ is lattice isometric to
$$
\left( \displaystyle\bigoplus_{n\in\mathbb{N}}\ell_n^1\right)_\infty:=\left\{x=(x_n)_n: x_n\in \ell_n^1 \textrm{~para cada~}n\in\mathbb{N} \textrm{~e~} \|x\|:=\displaystyle\sup_n\{\|x_n\|\}<\infty\right\},
$$
with the coordinatewise order, via the usual duality relation
$$(\varphi_j)_j \longmapsto (\varphi_j)_j\left((x_j)_j \right) = \displaystyle\sum_{j=1}^\infty \varphi_j(x_j). $$
It is plain that
\begin{align*}
T\colon \ell_1 \longrightarrow \left( \displaystyle\bigoplus_{n\in\mathbb{N}}\ell_n^1\right)_\infty~,~  T\left((x_j)_j\right) = \left((x_1),(x_1,x_2),(x_1,x_2,x_3),\ldots \right),
\end{align*}
is a linear isometric embedding. It is also a Riesz homomorphism: 
for $(x_j)_j$ and $(y_j)_j$ in $\ell_1$,
\begin{align*}
T((x_j)_j\vee(y_j)_j) &= T((x_j\vee y_j)_j)= ((x_1\vee y_1),(x_1\vee y_1,x_2\vee y_2), \ldots) \\
&= ((x_1\vee x_2), (x_1,x_2)\vee (y_1,y_2), \ldots)\\
&=((x_1),(x_1,x_2),\ldots)\vee ((y_1),(y_1,y_2),\ldots)= T((x_j)_j)\vee T((y_j)_j).
\end{align*}
It follows that $E^*$ contains a lattice copy of $\ell_1$, then Proposition  \ref{corquocbid} yields that $E$ is a closed ideal of $E^{**}$ and that the Banach lattices
$$ \left( \displaystyle\bigoplus_{n\in\mathbb{N}}\ell_n^\infty\right)_1^{**} {\rm ~and~}
\left( \displaystyle\bigoplus_{n\in\mathbb{N}}\ell_n^\infty\right)_1^{**}\Big/\left( \displaystyle\bigoplus_{n\in\mathbb{N}}\ell_n^\infty\right)_1
$$
lack the positive Schur property.
\end{example}


It is well known that $\ell_1$ is projectively universal for the class of separable Banach spaces in the sense that every separable Banach space is isometric to a quotient of $\ell_1$. The fact that the Schur property is preserved under isomorphisms and Theorem \ref{3latpro} give immediately the following.

\begin{corollary} \label{cor1} Let $E$ be a separable Banach space failing the Schur property. Then any closed subspace $M$ of $\ell_1$ such that $E$ is isomorphic to $\ell_1/M$ fails to be an ideal.
\end{corollary}

Although $\ell_1$ is a Banach lattice, it is not true that every separable Banach lattice is lattice isometric to a quotient of $\ell_1$ over a closed ideal. For example, $c_0$ is a separable Banach lattice that is not lattice isometric to a quotient of $\ell_1$ over a closed ideal (otherwise $c_0$ would be an AL-space because the quotient of an AL-space over a closed ideal is an AL-space \cite[p. 205]{aliprantis2}). In \cite{leung1}, Leung, Li, Oikhberg and Tursi constructed a separable Banach lattice $LLOT$ such that every separable Banach lattice is lattice isometric to a quotient of $LLOT$ over a closed ideal.

\begin{example}\rm
Since there are separable Banach lattices failing the positive Schur property and this property is preserved by lattice isomorphisms of Banach lattices, Theorem \ref{3latpro} gives that the Banach lattice $LLOT$ lacks the positive Schur property.
\end{example}

\begin{remark}\rm  In \cite{flores}, Flores, Hern\'andez, Spinu, Tradacete and  Troitsky introduced and developed the notion of $p$-disjointly homogeneous Banach lattices, $1 \leq p \leq \infty$. In \cite[Proposition 4.9]{flores} they proved that a Banach lattice is 1-disjointly homogeneous if and only if it has the positive Schur property. So, being a 1-disjointly homogeneous Banach lattice is a 3-lattice property.
\end{remark}

We finish this section with an application of our results to the dual positive Schur property, introduced by Aqzzouz, Elbour and Wickstead \cite{aqzzouz1} as follows: a Banach lattice $E$ has the {\it dual positive Schur property} if every weak$^*$-null sequence formed by positive functionals in $E^*$ is norm null.

Let us see first that the dual positive Schur property is not a 3-lattice property. According to Wnuk \cite[p. 768]{wnuk4}, $\ell_\infty$ and $\ell_\infty/c_0$ have the dual positive Schur property. But $c_0$ is a closed ideal of $\ell_\infty$ lacking the dual positive Schur property (the canonical unit vector form a weak$^*$-null non-norm null sequence in $\ell_1$). We have just mentioned that Wnuk \cite{wnuk4} established that the dual positive Schur property passes to quotients over closed ideals. So, all that is left to be established is the following:

\begin{proposition}\label{propdpsp} If the closed ideal $I$ of the Banach lattice $E$ and the quotient lattice $E/I$ have the dual positive Schur property, so has $E$.
\end{proposition}

\begin{proof} For a Banach space $X$, let us call $\sigma(X^*,X)$ the weak$^*$-topology on $X^*$ and $\|\cdot\|_{X^*}$ the norm topology on $X^*$. So,  
\begin{itemize}
\item $E$ has the dual Schur positive poperty if and only if $E^*$ has the ($\sigma(E^*,E), \|\cdot\|_{E^*} )$-PSP;
    \item $I$ has the dual Schur positive poperty if and only if $I^*$ has the $(\sigma(I^*,I),\|\cdot\|_{I^*})$-PSP;
    \item $E/I$ has the dual Schur positive poperty if and only if $(E/I)^*$ has the $(\sigma((E/I)^*,E/I),\|\cdot\|_{(E/I)^*})$-PSP.
\end{itemize}
Since the annihilator $I^\perp$ of  $I$ is a closed ideal of $E^*$ \cite[Corollary 1 of Proposition II.5.5]{schaefer1}, combining \cite[Theorems V.2.2 and V.2.3]{conway} and \cite[Corollary 1 of Proposition II.5.5]{schaefer1} we get

\begin{itemize}
    \item  $I^*$ has the $(\sigma(I^*,I),\|\cdot\|_{I^*})$-PSP if and only if $E^*/I^\perp$ has the $\left(\dot{[\sigma(E^*,E)]},\dot{\|\cdot\|}_{E^*}\right)$-PSP, where $\dot{[\sigma(E^*,E)]}$ is the quotient topology associated to the quotient operator $(E^*, \sigma(E^*,E)) \longrightarrow E^*/I^\perp$, and
    \item $(E/I)^*$ has the $(\sigma((E/I)^*,E/I),\|\cdot\|_{(E/I)^*})$-PSP if and only if $I^\perp$ has the $(\sigma(E^*,E),\|\cdot\|_{E^*})$-PSP.
\end{itemize}
The assumptions give that $I^\perp$ has the $(\sigma(E^*,E),\|\cdot\|_{E^*})$-PSP and that the quotient lattice $E^*/I^\perp$ has the $(\dot{[\sigma(E^*,E)]},\dot{\|\cdot\|}_{E})$-PSP, therefore Theorem \ref{implret} yields that $E^*$ has the $(\sigma(E^*,E),\|\cdot\|_{E^*})$-PSP, that is, $E$ has the dual positive Schur property.
\end{proof}


\section{The weak Dunford-Pettis property}

The weak Dunford-Pettis property was introduced by Leung \cite{leung} as follows: a Banach lattice $E$ has the {\it weak Dunford-Pettis property} if every weakly compact operator on $E$ sends weakly null sequences formed by pairwise disjoint vectors to norm null sequences. In \cite{castillo}, the authors mention that it was then unknown if this property is a 3-lattice property. As far as we know, this problem remains open, and the purpose of this short final section is to show that our methods can be used to give a contribution in this direction.

Following the approach of Gabriyelyan \cite{gabriyelyan}, we give the following definition in the context of Fr\'echet lattices.

\begin{definition}\rm A Fr\'echet lattice $E$ has the {\it sequential weak Dunford-Pettis property} if $\varphi_n(x_n)\longrightarrow 0$ whenever $(x_n)_n$ is a weakly null sequence in $E$ formed by positive vectors and $(\varphi_n)_n$ is a weakly null sequence in the strong dual $E^*_\beta$ of $E$.
\end{definition}

The next result was inspired by \cite[Proposition 6.6.c]{castillo} and its proof repeats some of the steps of the proof of Theorem \ref{implret}.

\begin{theorem}\label{lth}
Let $I$ be a closed ideal of the Fr\'echet lattice $(E,\tau)$. If $I$ has the sequential weak Dunford-Pettis property and $E/I$ has the $(\dot w_\tau,\dot \tau)$-PSP, then $E$ has the sequential weak Dunford-Pettis property.
\end{theorem}

\begin{proof} Given a weakly null sequence $(x_n)_n$ of positive vectors of $E$  and a weakly null sequence $(\varphi_n)_n$ in $E^*_\beta$, let $(\varphi_{n_j}(x_{n_j}))_j$ be an arbitrary subsequence of  $(\varphi_n(x_n))_n$. Considering the quotient operator $\pi\colon (E,\tau)\longrightarrow (E/I,\dot\tau)$, we have that $(\dot x_{n_j})_j=(\pi(x_{n_j}))_j$ is a sequence of positive vectors in $E/I$ and $\dot x_{n_j}\stackrel{\dot w_\tau}{\longrightarrow} 0$ in $E/I$. The $(\dot w_\tau, \dot\tau)$-PSP of $E/I$ gives $\dot x_{n_j}\stackrel{\dot \tau}{\longrightarrow}0$ in $E/I$ and reasoning similarly to the proof of Theorem \ref{implret} we can find a subsequence $(x_{n_{N_k}})_k$ of $(x_{n_j})_j$ and a sequence $(y_{N_k})_k$ of positive vectors of $I$ such that $x_{n_{N_k}}-y_{N_k}\stackrel{\tau}{\longrightarrow}0$. Hence,  $x_{n_{N_k}}-y_{N_k}\stackrel{w_\tau}{\longrightarrow}0$ in $E$ and from the linearity of the weak topology we get $$y_{N_k}=x_{n_{N_k}}-(x_{n_{N_k}}-y_{N_k})\stackrel{w_\tau}{\longrightarrow}0$$ in $E$, thus in $I$. Since $(\varphi_n)_n$ is weakly null in $E^*_\beta$, denoting by $\varphi_{n|_I}$ the restriction of each $\varphi_n$ to $I$ and applying  \cite[Theorem 8.11.3(d) and (c)]{narici} to the inclusion operator $(I, \tau) \longrightarrow (E, \tau)$ and its adjoint $E^*_\beta \longrightarrow I^*_\beta$ , we conclude that  $(\varphi_{n|_I})_n$ is weakly null in the strong dual $I^*_\beta$ of $I$. The sequential weak Dunford-Pettis property of $I$ gives $\varphi_{n_{{N_k}|_I}}(y_{N_k})\longrightarrow 0$, that is,  $\varphi_{n_{N_k}}(y_{N_k})\longrightarrow 0$.

On the other hand, the sequence $(\varphi_{n_{N_k}})_k$ is pointwise bounded because it is weakly null in $E^*_\beta$, therefore it is equicontinuous by the Banach-Steinhaus Theorem \cite[Theorem 11.9.1]{narici} (remember that $E$ is a Fr\'echet space). So, given $\varepsilon > 0$ there is a 0-neighborhood $U$ in $E$ such that $|\varphi_{n_{N_k}}(x)| < \varepsilon$ for every $x \in U$. Now the convergence $x_{n_{N_k}}-y_{N_k}\stackrel{\tau}{\longrightarrow}0$ implies that $\varphi_{n_{N_k}}(x_{n_{N_k}}-y_{N_k}) \longrightarrow 0$, therefore
 $$\varphi_{n_{N_k}}(x_{n_{N_k}})=\varphi_{n_{N_k}}(x_{n_{N_k}}-y_{N_k})+ \varphi_{n_{N_k}}(y_{N_k})\longrightarrow 0.$$  This proves that $\varphi_n(x_n)\longrightarrow 0$, and then $E$ has the sequential weak Dunford-Pettis property.
\end{proof}

\begin{corollary} Let $I$ be a closed ideal of the Banach lattice $E$. If $I$ has the weak Dunford-Pettis property and $E/I$ has the positive Schur property, then $E$ has the weak Dunford-Pettis property.
\end{corollary}

\begin{proof} Since the strong topology $\beta(E^*,E)$ on the dual $E^*$ of a Banach space $E$ coincides with the norm  topology \cite[Example 8.8.9]{narici}, the weak sequential Dunford-Pettis property coincides with the weak Dunford-Pettis property in a Banach lattice (see \cite[Corollary 2.6]{aqzzouz}), so the result follows from Theorem \ref{lth}. 
\end{proof}

\medskip

\noindent{\bf Acknowledgement.} The authors are grateful to Khazhak V. Navoyan for her helpful suggestions.

\bigskip

\noindent Faculdade de Matem\'atica~~~~~~~~~~~~~~~~~~~~~~Departamento de Matem\'atica\\
Universidade Federal de Uberl\^andia~~~~~~~~ IMECC-UNICAMP\\
38.400-902 -- Uberl\^andia -- Brazil~~~~~~~~~~~~ 13.083-859 - Campinas -- Brazil\\
e-mail: botelho@ufu.br ~~~~~~~~~~~~~~~~~~~~~~~~~e-mail: lucasvt09@hotmail.com

\end{document}